\documentclass[12pt,reqno]{amsart}
\usepackage[margin=1in]{geometry}

\pdfoutput=1

\usepackage{amsthm}
\usepackage{amssymb}
\usepackage{mathrsfs}
\usepackage{enumerate}
\usepackage{mathtools}
\usepackage[pdftex,dvipsnames,table,x11names]{xcolor}
\definecolor{darkred}{RGB}{139,0,0}
\definecolor{darkblue}{RGB}{0,0,139}
\definecolor{darkgreen}{RGB}{0,100,0}
\definecolor{darkmagenta}{RGB}{139,0,120}
\usepackage{enumitem}
\usepackage{tikz}
\usepackage{tikz-cd}
\usetikzlibrary{matrix,fit,calc,arrows,decorations.markings,3d,hobby,patterns}

\usepackage{booktabs}

\usepackage[linktocpage]{hyperref}
\hypersetup{
  colorlinks   = true,
  urlcolor     = darkred,
  linkcolor    = darkred,
  citecolor   = darkgreen}

\usepackage[noabbrev,capitalise]{cleveref}
\usepackage{adjustbox}

\setenumerate[1]{label=(\roman*),}

\setitemize[1]{label=\raisebox{0.25ex}{\tiny$\bullet$}}

\newif\ifarxiv{}

\newcommand{\myAcknowledgements}{The authors thank Jan Jendrysiak for helpful discussions. M.K. and P.S. acknowledge the Dagstuhl Seminar 23192 “Topological Data
Analysis and Applications” that initiated this collaboration.}
\newcommand{\myFunding}{Á.J.A. and M.K. research has been supported by the Austrian Science Fund (FWF) grant P 33765-N.}

\newcommand{\define}[1]{{\textit{#1}}}
\newcommand{\deff}[1]{{\define{#1}}}

\newcommand{\R}{\mathbb{R}}

\newcommand{\Rplus}{\mathbb{R}_+}
\newcommand{\N}{\mathbb{N}}

\newcommand{\eps}{\varepsilon}

\newcommand{\isomorphic}{\ensuremath{\cong}}

\usepackage{xspace}

\usepackage{mathrsfs}

\newcommand{\pointset}{S}

\DeclareMathOperator*{\diam}{diam}

\newcommand{\Prob}{\mathbb{P}}

\newcommand{\distribution}[1]{\mathrm{#1}}
\newcommand{\process}{\mathcal{P}}
\DeclareMathOperator{\Vol}{Vol}
\DeclarePairedDelimiter\abs{\lvert}{\rvert}
\DeclarePairedDelimiter\norm{\lVert}{\rVert}

\newcommand{\cech}{\mathcal{C}}
\newcommand{\rips}{\mathcal{R}}

\newcommand{\offset}{\mathcal{O}}
\newcommand{\degree}{\mathcal{D}}
\newcommand{\degreecech}{\mathcal{DC}}

\newcommand{\degreeoffset}{\degree\offset}
\newcommand{\cover}{\mathrm{Cov}}
\newcommand{\cross}{O}
\newcommand{\op}{\text{op}}

\newcommand\sbullet[1][.5]{\mathbin{\hbox{\scalebox{#1}{$\bullet$}}}}
\newcommand{\filt}{\sbullet[.5]}

\newcommand{\cat}[1]{\ensuremath{\mathsf{#1}}}

\newcommand{\cVect}{\cat{Vec}}
\newcommand{\cTop}{\cat{Top}}
\newcommand{\cSimp}{\cat{Simp}}

\providecommand\given{}
\newcommand\SetSymbol[1][]{%
  \nonscript\:#1\vert
  \allowbreak
  \nonscript\:
  \mathopen{}}
\DeclarePairedDelimiterX\Set[1]\{\}{%
  \renewcommand\given{\SetSymbol[]}
  #1
}

\newcommand\restr[2]{{%
  \left.\kern-\nulldelimiterspace %
  #1 %
  \vphantom{\big|} %
  \right|_{#2} %
  }}

\newcommand\restrFilt[2]{{%
  \left.\kern-\nulldelimiterspace %
  #1 %
  \vphantom{\big|} %
  \right|_{#2} %
  }}

\makeatletter
\newcommand*\rel@kern[1]{\kern#1\dimexpr\macc@kerna}
\newcommand*\widebar[1]{%
  \begingroup
  \def\mathaccent##1##2{%
    \rel@kern{0.8}%
    \overline{\rel@kern{-0.8}\macc@nucleus\rel@kern{0.2}}%
    \rel@kern{-0.2}%
  }%
  \macc@depth\@ne
  \let\math@bgroup\@empty \let\math@egroup\macc@set@skewchar
  \mathsurround\z@ \frozen@everymath{\mathgroup\macc@group\relax}%
  \macc@set@skewchar\relax
  \let\mathaccentV\macc@nested@a
  \macc@nested@a\relax111{#1}%
  \endgroup
}
\makeatother

\newcommand{\ignore}[1]{}

\newcommand{\tinymatrix}[1]{\begingroup\setlength\arraycolsep{2pt}\footnotesize
  \begin{pmatrix} #1\end{pmatrix}\endgroup}

\arxivtrue{}

\newtheorem{theorem}{Theorem}[section]
\newtheorem{lemma}[theorem]{Lemma}

\newtheorem{corollary}[theorem]{Corollary}
\newtheorem*{corollary*}{Corollary}
\theoremstyle{definition}
\newtheorem{definition}[theorem]{Definition}

\theoremstyle{remark}
\newtheorem{remark}[theorem]{Remark}

\makeatletter
\renewcommand\subparagraph{\@startsection{subparagraph}{5}%
  \z@{.5\linespacing\@plus.7\linespacing}{-.5em}%
  {\normalfont\bfseries}}
\makeatother

\title[Probabilistic Analysis of Multiparameter Persistence Decompositions]{Probabilistic Analysis of Multiparameter Persistence Decompositions into Intervals}
\author[Á.J. Alonso]{Ángel Javier Alonso}
\author[M. Kerber]{Michael Kerber}
\address{Institute of Geometry, Graz University of Technology, Austria}
\email{alonsohernandez@tugraz.at}
\email{kerber@tugraz.at}
\author[P. Skraba]{Primoz Skraba}
\address{Queen Mary University of London, United Kingdom}
\email{p.skraba@qmul.ac.uk}

\date{}

\begin{document}

\begin{abstract}
  Multiparameter persistence modules can be uniquely decomposed into
  indecomposable summands. Among these indecomposables, intervals stand out for
  their simplicity, making them preferable for their ease of interpretation
  in practical applications and their computational efficiency. Empirical
  observations indicate that modules that decompose into only intervals are
  rare. To support this observation, we show that for numerous common
  multiparameter constructions, such as density- or
  degree-Rips bifiltrations, and across a general category of point samples, the
  probability of the homology-induced persistence module decomposing into intervals
  goes to zero as the sample size goes to infinity.
\end{abstract}

\maketitle

\section{Introduction}
\subparagraph{Motivation.}
Persistence modules capture the topological evolution of data across
a range of parameters and are a major object of study
in topological data analysis (TDA).
To understand the structure of persistence modules, they are often split
into a direct sum of indecomposable elements. By the Krull-Remak-Schmidt theorem, such a decomposition is
unique up to isomorphism.
If the parameter space is one-dimensional, that is, the persistence module
is filtered over the real line, every indecomposable has the structure
of an interval, meaning that it represents a certain topological feature
in the data that is active in the range of scales given
by the interval boundaries. The collection of these intervals forms
the famous barcode of persistent homology.

In this work, we study persistence modules
over two real parameters.
The concept of intervals generalizes into this setting, but it is not
true that every 2-parameter persistence module decomposes into intervals.
In fact, the 2-parameter case is already of \emph{wild representation type},
meaning that indecomposables can become arbitrarily complicated.
Those modules that do admit a decomposition into intervals
are called \emph{interval decomposable}.

Restricting attention to interval-decomposable modules seems attractive:
an interval in the decomposition can be interpreted,
as in the one-parameter case, as a topological feature that persists
over a range of scales, providing a simple interpretation of these summands.
Moreover, this subclass of modules allows for faster algorithms for certain
problems, such as the computation of the bottleneck distance~\cite{dx-computing}.

The advantages of interval-decomposable modules raise the question of how
commonly we encounter this case in applications.
Unfortunately, experiments suggests that, typically,
persistence modules seem to contain non-interval summands;~see~\cite{ak-decomposition,hnox-refinement} for recent case studies. This leads to the motivating question for this paper:
can we capture this empirical evidence through a probabilistic mathematical statement?

\subparagraph{Contribution}
We show that many commonly used 2-parameter persistence modules
are unlikely to be interval decomposable when constructed over a large point
sample. More precisely, let $\process_n$ denote a Poisson point sample
in the $d$-dimensional unit cube
with an expected number of $n$ points
and fix $p\leq d-1$. Let $F(\process_n)$ denote a bifiltration constructed over $\process_n$,
and consider $H_p(F(\process_n))$, the induced persistence module in $p$-homology.
Then, we prove that, for common choices of the bifiltration $F$, the probability that $H_p(F(\process_n))$ is interval-decomposable
approaches zero as $n\rightarrow\infty$. In particular, we take $F$ to be the
sublevel offset (or \v{C}ech or Vietoris-Rips) bifiltration, with random
function values, or given by a kernel density estimate or a fixed function,
the degree offset (or \v{C}ech or Vietoris-Rips) bifiltration, or the multicover
bifiltration (with $p = 0$); see the text for the detailed statements.

The proof consists of two major parts: we first show that
for any fixed point set $S$ with a constant number of points in the unit cube,
the Poisson point sample $\process_n$ contains an approximate scaled copy
of $S$ in some subcube with probability going to $1$ as $n\rightarrow\infty$.
This result follows from basic properties of Poisson point processes
and are straightforward for those familiar with stochastic geometry.
However, we believe that our exposition is of benefit for the TDA community.

The second part is to identify point patterns that lead to non-intervals
in the induced persistence module.
The definition of this pattern depends on the bifiltration $F$ chosen, so a separate proof
is required for all aforementioned choices of $F$.
Importantly, these patterns have to be stable under slight perturbation of the point coordinates
to use the probabilistic result from the first part.
We aim for small point patterns
and achieve the provably minimal cardinality of points in many examples. Small point patterns
are more likely to be realized for smaller samples, so our examples facilitate a
quantified analysis of interval-decomposability for $\process_n$ with concrete values of $n$
(which is not subject of this paper). On a technical level, we show the presence of non-intervals
by restricting attention to a finite subposet of $\R^2$ and show that an indecomposable
non-interval summand is present. This technique together with the small size
of the point patters leads to short and pictorial proofs of our main theorems.

\subparagraph{Related work.}
Perhaps closest to our result is the paper by Buchet and Escolar~\cite{be-realizations}.
They provide a simple construction of a simplicial complex
that yields indecomposables of arbitrary large degree, and the construction
is stable under small perturbations. This is partially stronger than what
we provide, as our construction only guarantees indecomposables of degree $2$.
However, their complexes are only shown to be realized via a dynamic
point process in Euclidean space and does not directly apply to
the bifiltrations considered in our work.

Also closely related is the work by Alonso and Kerber~\cite{ak-decomposition}.
Adapted to our notation, they show that
for $F$ the sublevel offset bifiltration and $p=0$,
at least a quarter of the indecomposables of $H_p(F(X_n))$ will be intervals
in expectation when $n$ goes to infinity,
with $X_n$ being a point process with rather weak assumptions.
While the lower bound of $1/4$ might not be tight (as suggested by
experimental evidence in their paper), our result complements their result,
saying that we cannot expect all indecomposables to be intervals.

Recently, Hiraoka et al.~\cite{hnox-refinement} do a case study
on the indecomposables of a commutative ladder, a special case of bifiltrations.
Their experiments confirm the observation that while intervals are frequent,
non-intervals typically appear in the decomposition (Figure~17 in~\cite{hnox-refinement}).

In the purely algebraic setting, Bauer and Scoccola~\cite{bauer2023generic} have
recently shown that being nearly indecomposable (under the interleaving
distance) is a generic property of multiparameter persistence modules.

The bifiltrations considered in this work are the standard models for distance- and density-based bifiltrations
in the literature; they are discussed, for instance, in the recent survey by Botnan and Lesnick~\cite{bl-introduction}
and investigated regarding their stability properties by Blumberg and Lesnick~\cite{bl-stability}.
Efficient algorithms to compute such bifiltrations are an active area of research, with recent results
for sublevel offset bifiltrations~\cite{akll-delaunay}, for bifiltrations of clique complexes
(including Vietoris-Rips complexes)~\cite{akp-filtration} and for multicover~\cite{bdk-sparse,cklo-computing}.
The meataxe algorithm~\cite{parker-computer} provides
a method to compute the decomposition of the induced persistence module.
That algorithm works for more generalized setups; a specialized algorithm for persistence modules
has been proposed by Dey and Xin~\cite{dx-generalized}.
An important intermediate step from bifiltrations to their decomposition is the computation
of a minimal presentation, for which efficient algorithms also exist~\cite{fkr-compression,lw-computing,dx-computing}.

The study of homology in random settings is also an active topic of research;
see~\cite{bobrowski2018topology} for a survey of the geometric setting. The
homology of random geometric complexes at a single scale has been extensively
studied, e.g.~\cite{kahle2011random,bobrowski2011distance,
  yogeshwaran2015topology}. Much less is known about the \emph{persistent}
homology of random geometric complexes, but the last few years have seen
significant progress, including limit theorems~\cite{duy2016limit} and the
expected maximal persistence~\cite{bobrowski2017maximally}, with recent
extensions to multiparameter persistence~\cite{botnan2022consistency}. The
technique of exhibiting a point set that appears with positive probability and
has a given property is a standard approach, and is often used to show bounds in
random settings; the work of Kahle~\cite{kahle2011random}, which proves lower
bounds on the radius at which the $k$-dimensional Betti number is non-zero in
Vietoris-Rips complexes, is an example of this.

\ifarxiv{}
\subparagraph{Acknowledgements.} \myAcknowledgements{} \myFunding{}
\else{}
\fi{}

\section{Ball configurations and Poisson processes}\label{sec:probability}
We refer the reader to \cite{last2017lectures} for a reference on the basic facts below.
The \emph{Poisson distribution} $\distribution{Poisson}(\lambda)$ with rate $\lambda$
is a discrete random variable $X$ with
\[
\Prob(X=k)=\frac{\lambda^k}{k!}e^{-\lambda}.
\]
For $n\in\N$, a \emph{Poisson point sample} $\process_{n}$ over a
measurable\footnote{More precisely, the set must be measurable with respect to
  the Lebesgue measure, however in our case the sets are nice so we generally omit the qualifier.} set $X$ is obtained
by sampling a natural number $N$ from $\distribution{Poisson}(n)$ and then sampling
$N$ points $(x_1,\ldots,x_N)$ independently and uniformly at random from $X$.
In this work, we will restrict to the case that $X$ is the $d$-dimensional unit cube $[0,1]^{d}$.

A Poisson point sample has two important two properties: for every set
$A\subset [0,1]^d$,
the random variable $\abs{\process_{n}\cap A}$ is a Poisson distribution
with rate $n\Vol(A)$,
and the sample satisfies \textit{spatial independence}: for disjoint sets $A,B\subset[0,1]^d$,
$\abs{A\cap\process_{n}}$ and $\abs{B\cap\process_{n}}$ are independent
random variables. See Definition 3.1 in \cite{last2017lectures}.

\begin{figure}
\centering
\includegraphics[width=6cm]{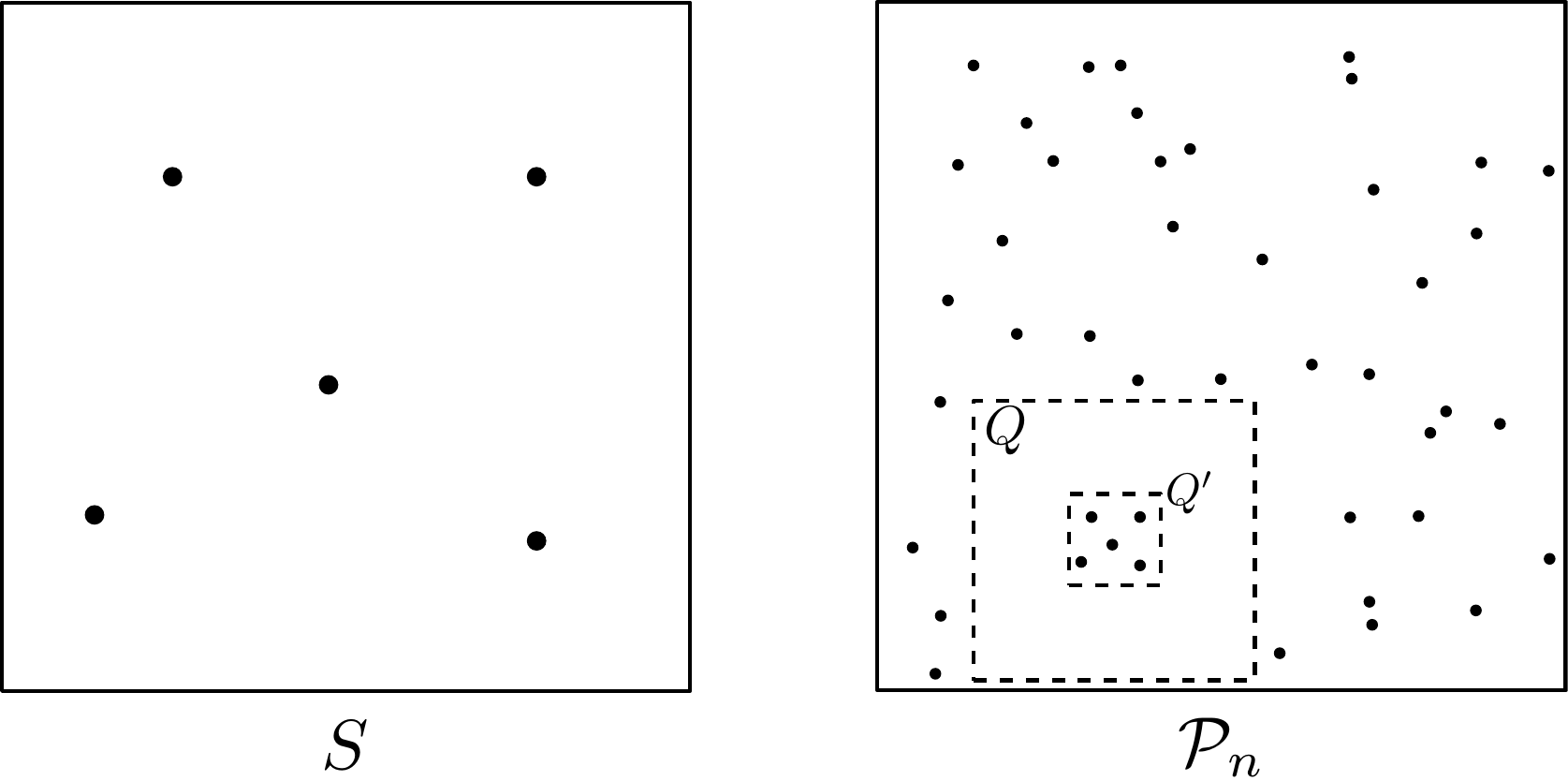}
\caption{Left: A point set $S$ of $5$ points. Right: A Poisson point sample
with a subcube that contains an isolated scaled copy of $S$.}\label{fig:poisson_illu}
\end{figure}

Fix now an arbitrary point set $S$ of constant size in $(0,1)^d$ and $\eps>0$.
We want to show  that for $n$ large enough, it is likely that $\process_{n}$
contains a scaled version of the configuration $S$ up to a perturbation
of every point by at most $\eps$; see Figure~\ref{fig:poisson_illu}
for an illustration.

We define the property of containing a scaled version of $S$ formally. Fix
$S=(a_1,\ldots,a_m) \in [0,1]^d$, and let $B_1,\ldots,B_m$ denote the balls
centered at $a_1,\ldots,a_m$ with radius $\eps$. We assume for simplicity that
$\eps$ is sufficiently small, so that $B_1,\ldots,B_m$ are pairwise disjoint and
every $B_i$ is contained in the unit cube.

Next, fix a realization of the Poisson point process $\process_n$. For a fixed subcube $Q\subset[0,1]^d$,
let $\alpha$ denote the side length of $Q$ and $x$ its center.
Furthermore, let $Q'$ denote the subcube of $Q$ with same center and side length $\frac{\alpha}{4\sqrt{d}}$.

There is a canonical axis-preserving bijection that maps $[0,1]^d$ to $Q'$.
This bijection maps $a_1,\ldots,a_m$ to points in $Q'$, and the $\eps$-balls $B_1,\ldots,B_m$
to pairwise disjoint balls $B_1',\ldots,B_m'$ of radius $\frac{\alpha\eps}{4\sqrt{d}}$.
We say that $\process_n$ \deff{contains a scaled $\eps$-copy of $S$ in $Q$}, if
$\process_n$ contains exactly one point
in $B_i'$, for every $i=1,\ldots,m$, and $\process_n$ contains no further point in $Q$.
We say that $\process_n$ \deff{contains a scaled $\eps$-copy of $S$} if there exists a subcube $Q$ such that
$\process_n$ contains a scaled $\eps$-copy of $S$ in $Q$. We note that the size of the subcube $Q'$ and the condition that $Q$
contains no other points ensures that the points in $Q'$
are ``well-separated'' from the remaining points in $\process_n$. Specifically, the distance between any point in $Q'$ and any point outside $Q$ is at least the diameter of $Q'$. See Figure~\ref{fig:poisson_illu} for an example.

\begin{theorem}\label{thm:poisson}
  For every finite point set $S$ and every $\eps>0$ small enough, there exists a constant $\alpha$ independent of $n$ such that
    \begin{equation*}
    \Prob(\text{$\process_n$ contains a scaled $\eps$-copy of $S$}) \geq  1-e^{-\alpha n}.
   \end{equation*}
 \end{theorem}
\begin{proof}
  We consider a subcube $Q$ of side length $n^{-1/d}$. %
  The scaled balls $B_1',\ldots,B_m'$ then have radius
  $\frac{\eps n^{-1/d}}{4\sqrt{d}}$ and, hence, volume
  $\frac{\eps^d n^{-1}}{(4\sqrt{d})^{d}} \omega_d$, where $\omega_d$ is the volume of
  the $d$-dimensional unit ball. By the first characteristic property of Poisson
  point samples, the number of points in $\process_n\cap B_i'$ is a Poisson
  distribution with rate $\frac{\eps^d}{(4\sqrt{d})^{d}} \omega_d$, and therefore is
  independent of $n$. Thus, the probability that $B_i'$ contains exactly one
  point of $P_n$ is some constant $\lambda_1$ that is also independent of $n$.

Set $T\coloneqq Q\setminus(B_1'\cup\ldots\cup B_m')$ as the complement of the scaled $\eps$-balls in $Q$.
For $\process_n$ to contain a scaled $\eps$-copy of $S$ in $Q$, $T$ must not contain a point.
The probability for this to occur is lower bounded by the probability that $Q$ does not contain a point.
The number of points in $\process_n\cap Q$ is a Poisson distribution with rate $1$
(since the volume of $Q$ is $1/n$). So, the probability for no points in $\process_n\cap T$ is lower-bounded
by a constant $\lambda_0$ independent of $n$.

$\process_n$ contains a scaled copy of $S$ in $Q$ if and only if each $B_i'$
contains exactly one point of $\process_n$, and $T$ contains no point of $\process_n$.
By spatial independence, we have that these events are independent, and,
therefore, the probability that $\process_n$ contains a scaled copy of $S$ in
$Q$ is at least $\lambda=\lambda_0\cdot(\lambda_1)^{m}$, which is a constant
independent of $n$.

To complete the proof, note that we can pack at least $\lfloor{n^{1/d}\rfloor^d}$ disjoint $d$-dimensional subcubes of side length $n^{-1/d}$ in the unit cube
without overlap, which is larger than $\frac{1}{2}n$ for $n$ large enough. 
In each of these cubes, the probability of not containing a scaled copy of $S$ is at most $1-\lambda$. Again by spatial independence,
the probability that none of the subcubes contain a scaled copy of $S$ is at most $(1-\lambda)^{\frac{1}{2}n}$ which may be upper bounded by $e^{-\alpha n}$ for some constant $\alpha$.
\end{proof}

The above is for \textit{homogenous} Poisson processes, but below we show that the result
holds for non-homogeneous Poisson processes as well. A non-homogeneous Poisson
point process is determined by an (integrable) intensity function $f$. In
practice, $f$ is a probability density function. The number of points in any set
$A\subset \R^d$ is a Poisson random variable with rate $n\cdot\int_A f(x) dx$.
We note that this type of process retains the property that the number of points
(and point configurations) in any countable number of fixed disjoint sets are
independent.
\ifarxiv
The proof, analogous to the one of Theorem~\ref{thm:poisson},
is included in~\cref{sec:proof_non_homogeneous}.
\else
The proof, analogous to the one of Theorem~\ref{thm:poisson},
is included in the full version.
\fi

\begin{corollary}\label{cor:poisson}
Let $\process^f_{n}$ be a non-homogeneous Poisson point process with a continuous intensity $n\cdot f(x)$ for $x\in \R^d$.  If for some $p\in \R^d$ and constant $\delta>0$, there exists a cube $p+[0,\delta]^d$  where
$f$ is strictly positive, then for every finite finite point set $S$ and every $\eps>0$ small enough, there exists a constant $\alpha$ independent of $n$ such that
\begin{equation*}
\Prob(\text{$\process^f_n$ contains a scaled $\eps$-copy of $S$}) \geq 1-e^{-\alpha n}.
\end{equation*}
\end{corollary}
\section{Decompositions of persistence modules}
\subparagraph{Persistence modules.}
For a partially ordered set (poset) $P$, a \define{persistence module} $M$ over
$P$ is a functor from $P$ to $\cVect$, the category of (for us,
finite-dimensional) vector spaces over a fixed field $K$. This means that a persistence module assigns to each
$p\in P$ a vector space $M_p$ and to any two $p\leq q$ a linear map
$M_{p\to q}:M_p\to M_q$, such that $M_{p\to s} = M_{q\to s}\circ M_{p\to q}$ for
all $p\leq q\leq s$, and $M_{p\to p}$ is the identity. Morphisms and isomorphisms are defined as usual for functor
categories.

In this work, we only consider product posets of $\Rplus$ and $\Rplus^{\op}$,
where $\Rplus:=[0,\infty)$ and $\Rplus^{\op}$ is the opposite poset, and finite
subposets of these products.

\subparagraph{Decomposition.}
For two persistence modules $M$, $N$ over a common poset $P$, its \emph{direct
  sum} $M\oplus N$ is the persistence module defined by taking direct sums pointwise:
for $p\in P$, we take $M_p\oplus N_p$, and, for every $p\leq q$, the
map $M_p\oplus N_p\to M_q\oplus N_q$,
$(x,y)\mapsto (M_{p\to q}(x),N_{p\to q}(y))$.

A module $M$ is called \emph{indecomposable}
if $M \isomorphic A\oplus B$ implies that $A= 0$ or
$B=0$, where $0$ is the persistence module for which all vector spaces are $0$.
All the persistence modules we consider, which come from finite point sets, can be
decomposed uniquely, up to isomorphism, as a direct sum of indecomposable
summands by the Krull-Remak-Schmidt theorem.

\subparagraph{Restrictions.}
Given a subposet $Q\subset P$,
a persistence module $M$ over $P$ induces a persistence module over $Q$ in the natural
way: taking the vector spaces and linear maps of $M$ at the places of $Q$.
We call this persistence module $\restrFilt{M}{Q}$ over $Q$ the \deff{restriction} of $M$.
Moreover, if $M\isomorphic A\oplus B$ is a decomposition, we have $\restrFilt{M}{Q} \isomorphic \restrFilt{A}{Q}\oplus \restrFilt{B}{Q}$.

A persistence module $M$ over a poset $P$ is \emph{thin} if $\dim M_p\leq 1$ for all $p\in P$.
Clearly, if $M$ is a thin persistence module over $\Rplus^2$, the restricted module $\restrFilt{M}{Q}$
is thin as well for any subposet $Q\subset\Rplus^2$.
A persistence module $M$ is called \emph{interval decomposable}, if it admits a decomposition
into thin modules (this is not the usual definition, see below). We obtain:
\begin{lemma}\label{lem:thin-decomp}
  If a persistence module $M$ over $\Rplus^{2}$ is interval decomposable, then
  for any subposet $Q\subset\Rplus^{2}$ the restriction $\restrFilt{M}{Q}$ is
  interval decomposable.
\end{lemma}
One usually defines interval decomposability as decomposition into interval
modules, a subcase of thin indecomposables whose support is convex and connected
(e.g., \cite[Def 2.1]{bl-introduction}). However, for all persistence modules we
consider in this paper, the two notions coincide: every thin module can be
decomposed into (finitely many) interval modules~\cite[Thm 24]{abeny-interval}.

\subparagraph{Two non-thin indecomposables.}
As an interface to the examples in the next section, we consider the two modules
depicted in Figure~\ref{fig:finite_posets}. To check that they are indecomposable is
elementary: For the right hand side module $M$, suppose that
$M\isomorphic A\oplus B$ and note that if $A$ is non-zero at the source vertex
of the underlying graph, then it is also not zero at the sinks of the graph, and
thus necessarily $B = 0$. The left example can be similarly checked.

\begin{figure}
\centering
    $\begin{tikzcd}[ampersand replacement=\&,row sep=15, column sep=15pt]
      K   \rar{\tinymatrix{1 \\ 1}} \& K^{2} \rar{\tinymatrix{0 & 1}}    \& K  \\
      \& K     \uar{\tinymatrix{0 \\ 1}}    \&
    \end{tikzcd}$
    \quad\quad
    $\begin{tikzcd}[ampersand replacement=\&,row sep=25pt, column sep=25pt]
      K \& \\
      K^{2} \uar{\tinymatrix{0 & 1}}\rar{\tinymatrix{1 & 0}} \& K \\
      K \uar{\tinymatrix{1 \\ 1}}
    \end{tikzcd}$
\caption{Two non-thin indecomposable persistence modules over finite posets. Both posets are subposets of $\Rplus^2$.}\label{fig:finite_posets}
\end{figure}

We will use the following frequently. It follows immediately
from~\cref{lem:thin-decomp}.

\begin{lemma}\label{thm:non-interval-crit}
  Let $M$ be a persistence module over $\Rplus^2$ and $P$ be a subposet of $\Rplus^2$
  such that $\restrFilt{M}{P} \isomorphic A\oplus B$, where $A$ is one of modules
  in~\cref{fig:finite_posets}. Then, $M$ is not interval decomposable.
\end{lemma}

\section{Non-interval decomposability with high probability}

A \define{bifiltration} $F$ is a functor from the poset $\Rplus^{2}$ to the
category $\cTop$ of topological spaces or the category $\cSimp$ of simplicial
complexes. Writing $H_k\colon\cTop\to \cVect$ for the $k$-homology functor
with coefficient in the base field $K$,
$H_k(F):\Rplus^2\to \cVect$ is a persistence module.

We are interested in properties of point sets that are preserved under
perturbation.
We say that a finite point set $S'\subset \R^{d}$ is an
\deff{$\varepsilon$-perturbation} of another point set $S\subset\R^{d}$ if there
exists a bijection $f\colon S'\to S$ such that $\norm{x - f(x)} \leq \varepsilon$.

Below, we go over typical bifiltrations based on point sets. We argue the same
way for each type of bifiltration: we first show that there
exists a point set $S$ whose bifiltration induces a non-interval-decomposable
persistence module, and that this is preserved by $\varepsilon$-perturbation. We
then conclude that the Poisson process $\process_{n}$ of~\cref{sec:probability}
has the same property with high probability precisely because $S$ exists.  Note that with high probability means with probability greater than $1-\frac{1}{n^c}$ with $c>0$.

\subsection{Offset bifiltrations}\label{sec:cech}

The \deff{offset (or union-of-balls) filtration}
$\offset_{\filt}(\pointset)\colon\Rplus\to\cTop$ of a point set $S\subset\R^{d}$, is given by setting
$\offset_r(\pointset)\coloneqq\bigcup_{x\in\pointset} B_r(x)$, where $B_r(x)$ is
the closed ball with center $x$ and radius $r$. The \deff{\v{C}ech filtration}
$\cech_{\filt}(\pointset)\colon\Rplus\to\cSimp$ is the \textit{nerve} of $\offset_{\filt}(\pointset)$,
that is,
$\cech_{r}(\pointset) = \Set{\sigma\subset \pointset\given \bigcap_{x\in\sigma}B_{r}(x)\neq \varnothing}$,
or, equivalently, a simplex $\sigma\subset S$ is in
$\cech_{r}(S)$ if the minimum enclosing ball of $\sigma$ has radius less than or
equal to $r$.
Given a function $\gamma\colon\pointset\to\Rplus$, we define the
\deff{sublevel offset bifiltration} $\offset_{\filt}(\gamma)$ of $S$ with respect
to $\gamma$ by taking
\[
  \offset_{r,s}(\gamma):= \offset_{r}(\gamma^{-1}([0, s])).
\]
And analogously for the \deff{sublevel \v{C}ech bifiltration} $\cech_{\filt}(\gamma)$.
As a consequence of the nerve theorem~\cite{bauerUnifiedViewFunctorial2023a},
$H_{k}(\offset_{\filt}(\gamma))$ and $H_{k}(\cech_{\filt}(\gamma))$ are isomorphic
persistence modules for any $k$.

For the rest of the subsection, we fix $d>1$ to be a fixed constant and $1\leq k\leq d-1$
as the homology dimension; the case $k=0$ will be handled separately in Section~\ref{sec:zero_dim}.

\begin{theorem}\label{thm:non_interval_cech_random}
  Let $\process_n$ be a Poisson point process in $\R^d$. For each point $x$ in $\process_n$, assign
  $\gamma(x)$ uniformly at random in $[0,1]$. Then, the persistence module
  $H_k(\offset_{\filt}(\gamma))$ is not interval decomposable
  with high probability.
\end{theorem}
We follow the strategy hinted at the start of the section. The first step is the following:

\begin{lemma}\label{lem:existence_cech_non_interval}
  There is a finite point set $\pointset\subset\R^{d}$ of $d+2$ points and a
  function $\gamma\colon \pointset\to\R$ such that
  $H_{d-1}(\offset_{\filt}(\gamma))$ (or, equivalently,
  $H_{d-1}(\cech_{\filt}(\gamma))$) is not interval decomposable.

  Moreover, the same holds for any $\varepsilon$-perturbation $S'\subset\R^{d}$
  of $S$, for a small enough $\varepsilon > 0$.
\end{lemma}

\begin{proof}
  Let $\Sigma$ be a regular $d$-simplex with unit edge length in $\R^{d}$. Fix a
  facet $\sigma$ of $\Sigma$. For $\delta>0$, we denote by $\Sigma_{\delta}$ a
  perturbed version of $\Sigma$ given by moving the vertex opposite to $\sigma$
  perpendicularly towards $\sigma$ by a distance of $\delta$. Define $\Sigma^{-}$ as the
  regular $d$-simplex obtained by reflecting $\Sigma$ along the hyperplane of
  $\sigma$, and define $\Sigma^{-}_{\delta'}$ as its perturbed version in the
  analogous way, using $\delta'<\delta$. See~\cref{fig:regular_simplex}.
  \begin{figure}
    \centering \includegraphics{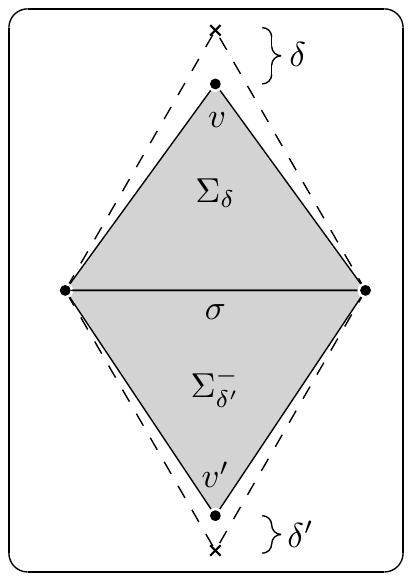}
    \caption{The construction of~\cref{lem:existence_cech_non_interval} for
      $\R^{2}$: two perturbed equilateral triangles glued along
      $\sigma$.}\label{fig:regular_simplex}
  \end{figure}

  We take $S$ to be the vertex set of $\Sigma_{\delta}\cup\Sigma^{-}_{\delta'}$,
  for small enough $\delta' < \delta$ such that the circumcenter of
  $\Sigma_{\delta}$ is still in the interior of $\Sigma_{\delta}$.
  Denote by $v$ the vertex in $\Sigma_{\delta}$ opposite to $\sigma$, and let
  $\gamma\colon S\to\R$ be any function such that $\gamma(v) > \gamma(x)$ for
  any other $x\in S$. We argue that $H_{d-1}(\cech_{\filt}(\gamma))$ is not interval
  decomposable by looking at a restriction of it.
  The minimum enclosing ball (given by its circumscribed hypersphere) of
  $\Sigma$ has radius $R \coloneqq \sqrt{\frac{d}{2(d+1)}}$, and the minimum
  enclosing ball of its facets has radius $r\coloneqq\sqrt{\frac{d-1}{2d}}$
  (see, for instance,~\cite[Theorem 4.5.1]{fiedlerMatricesGraphsGeometry}). Let $r_{\delta}$ be
  the radius of the minimum enclosing ball of the facets that are not $\sigma$
  of $\Sigma_{\delta}$. Let $R_{\delta}$ be the minimum enclosing ball radius of
  $\Sigma_{\delta}$. By the way we construct $S$, we have
  \begin{equation*}
    r_{\delta} < r_{\delta'} < r \text{, and } R_{\delta} < R_{\delta'} < R.
  \end{equation*}
  Finally, let $s$ be the maximum value of $\gamma$ in $S\setminus\Set{v}$.

  We look at the persistence module $H_{d-1}(\cech_{\filt}(\gamma))$ restricted to the
  finite subposet $P\subset\Rplus^{2}$ given by
  $(r_{\delta'}, \gamma(v)) < (r, \gamma(v)) < (R_{\delta}, \gamma(v))$ and
  $(r, s)$, see~\cref{fig:regular_simplex_filtration}. Note that
  $\cech(\gamma)_{(r_{\delta'}, \gamma(v))}$ consists of all $(d-1)$-simplices
  of $\Sigma_{\delta}$ and $\Sigma^{-}_{\delta'}$ except $\sigma$, $\cech(\gamma)_{(r, \gamma(v))}$ consists of all
  $(d-1)$-simplices, while the only $d$-simplex in
  $\cech(\gamma)_{(R_{\delta}, \gamma(v))}$ is $\Sigma_{\delta}$. All in all, the
  restricted persistence module can be seen to be isomorphic to the left example
  of~\cref{fig:finite_posets}, and~\cref{thm:non-interval-crit} yields the
  desired statement.
  \begin{figure}
    \centering \includegraphics{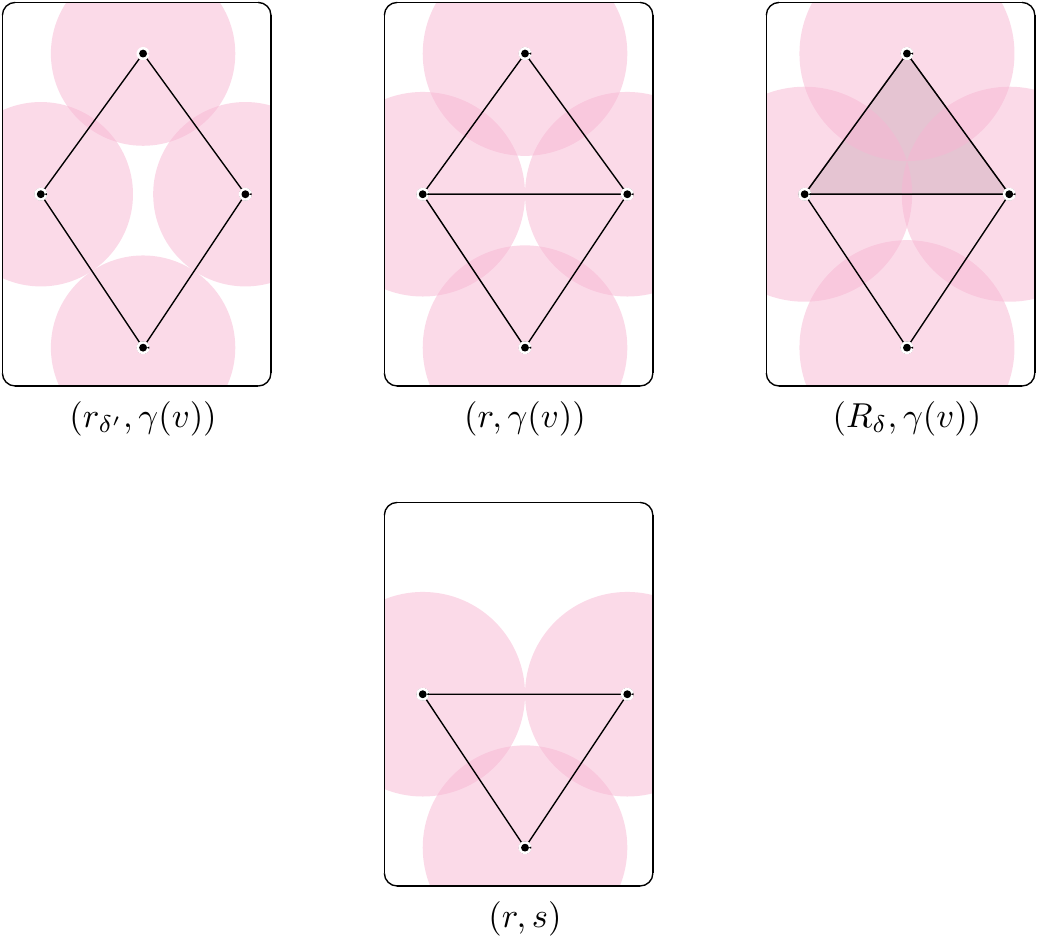}
    \caption{The \v{C}ech and offset bifiltration
      of~\cref{lem:existence_cech_non_interval} restricted to the subposet
      $P\subset\Rplus^{2}$.}\label{fig:regular_simplex_filtration}
  \end{figure}
  Finally, note that there exists a sufficiently small $\varepsilon > 0$ such that
  any $\varepsilon$-perturbation of $S$ induces a non-interval decomposable
  persistence module.
\end{proof}

\begin{remark}
  The example above is minimal: if the $(d-1)$-homology of an offset
  filtration is not a thin persistence module, then the underlying point set has
  at least $d+2$ points.
\end{remark}
    
\begin{proof}[Proof of~\cref{thm:non_interval_cech_random}]
  Given homological degree $k$ and dimension $d$, set $\tilde{d}=k+1\leq d$.  Let $\pointset\subset\R^{\tilde{d}}$ be the construction
  of~\cref{lem:existence_cech_non_interval}. We adopt the notation used in the
  proof of that lemma. We scale down $S$ and embed it in $[0,1]^{d}$. By the triangle inequality, if $S$ is stable under any $\varepsilon$-perturbation in $\R^{\tilde{d}}$, it is stable under any $\frac{\varepsilon}{2}$-perturbation in $\R^d$.  
  \cref{thm:poisson} applied to $\pointset$ yields that $\process_{n}$ contains
  an (isolated) scaled $\frac{\varepsilon}{2}$-copy of $S$ with high probability. The only requirement on the
  function $\gamma$ of~\cref{lem:existence_cech_non_interval} is that it has a
  unique maximum at $v\in S$, as defined in the proof of the lemma. Since $\gamma(x)$
  is sampled uniformly, this occurs with probability
  $\frac{1}{\tilde{d}+2}$ (since $|S|=\tilde{d}+2$) which is a constant independent of $n$. It follows that with high
  probability, $\process_n$ contains a scaled $\frac{\eps}{2}$-copy of $S$ where $v$ has
maximal $\gamma$-value.

  Let $S'$ be the scaled $\frac{\eps}{2}$-copy of $S$ contained in $\process_{n}$.
  Consider the subposet $P\subset\Rplus^{2}$ given in the proof
  of~\cref{lem:existence_cech_non_interval}, and change its values so that the
  construction of the lemma applies to the scaled-down copy $S'$, obtaining $P'\subset\Rplus^{2}$.
  \cref{thm:poisson} guarantees that the distance between a point in $S'$ and a
  point of $\process_{n}$ not in $S'$ is at least the diameter of $S'$. It
  follows that, at every $p\in P'$, the offset of
  $S'$, $\offset_{p}(\restr{\gamma}{S'})$, is an isolated connected component in
  $\offset_{p}(\gamma)$. It follows that
  $H_{\tilde{d}-1}(\offset_{\filt}(\restr{\gamma}{S'}))$, restricted to $P'$, is a
  summand in the decomposition of $H_{\tilde{d}-1}(\offset_{\filt}(\gamma))$,
  restricted to $P'$. Since this summand is not interval decomposable and since
  $\tilde{d} = k + 1$, we conclude that $H_{k}(\offset_{\filt}(\gamma))$ is not
  interval decomposable, as required.
\end{proof}

\subparagraph{Probability density functions and density estimation.} The
function $\gamma$ of the sublevel offset bifiltration $\offset_{\filt}(\gamma)$
is usually given by the application itself or
taken so that $\gamma$ captures the \textit{density} of the points: lower values
of $\gamma$ mean points of lower density, and thus more likely to be
noise~\cite{carlssonTheoryMultidimensionalPersistence2009}. Therefore, we are usually more
interested in regions of higher density. Note that when working with
sublevel sets, as in $\offset_{\filt}(\gamma)$, we would need to invert the order
of the densities---our methods work either way.

If $\gamma$ captures the density, we consider a
Poisson process with intensity  $\gamma$,
$\process_{n}^{\gamma}$, and then we would ideally use $\gamma$ itself in
$\offset_{\filt}(\gamma)$. However, $\gamma$ must usually be estimated. We first treat the case of estimating the density,
in~\cref{thm:process_estimation} below, and then, later in~\cref{thm:process_gamma_fixed}, we treat the more
technical case of a fixed $\gamma$, under generic smooth conditions.

A common approach to density estimation is to use a \emph{kernel}~\cite{silvermanDensityEstimationStatistics1986}. Choosing a kernel
$K_h$, where $h$ is the \emph{bandwidth}, the estimated density at a point $p$ is
\[ \hat{\gamma}(p)= \frac{1}{|S|} \sum_{q\in S} K_h(d(p,q)),\]
where $S$ is the point set.  We note that the kernel must
satisfy certain conditions but, since we do not use them, we refer the reader to any text on statistics e.g. \cite{wasserman2004all}.
The most basic kernel is a \deff{ball kernel}, $K_{h}(x) = 1$ if $x\leq h$; the
estimated density is
\[ \hat{\gamma}(p) = \frac{\abs{S \cap B_h(p)}}{\abs{S}},\] where $B_{h}(p)$ is
the closed ball of radius $h$. This is the kernel we will now work with. We believe
our methods work for more general choices of kernel, but the probabilistic
elements of the proof become much more delicate; we comment on this in the
conclusion.
\begin{theorem}\label{thm:process_estimation}
  Let $\process^\gamma_n \in \R^d$ be a Poisson point process with intensity $\gamma$, where  $\gamma$ is a Morse function. If $\hat{\gamma}$ is the ball estimator with $h=(\log n/n)^{1/d}$, the persistence module
$H_k(\offset_{\filt}(\hat{\gamma}))$ is not interval decomposable with high probability.
\end{theorem}
\begin{proof}
  The proof is similar to Theorem \ref{thm:non_interval_cech_random}. The main
  difference is we require the kernel estimates to be independent for each
  subcube $Q_i$. That is for any $p\in Q_i$ and $q\in Q_j$, the distance between
  $p$ and $q$ must be at least twice the bandwidth. Therefore in the proof of
  Corollary \ref{cor:poisson} (resp. the proof of Theorem 1), we can only pack
  $(\lfloor n/4^d\log n)^{1/d}\rfloor)^d = O(n/\log n)$ subcubes into the unit cube
  such that the estimator supports are disjoint. %

  It remains to analyze the case for one subcube. We again fix a point set $S$ and $\eps$ which leads to a non-interval decomposition. While the estimates for each sample are dependent, we show that the probability of any permutation of ordering is strictly positive. For any two points in $s_1,s_2\in S$, consider any points $p\in B_\eps(s_1)$ and $q\in B_\eps(s_2)$. Define the sets $A_p = B_h(p) - B_h(q)\cap B_h(p)$ and $A_q = B_h(q) - B_h(p)\cap B_h(q)$. By construction, $A_p$ and $A_q$ of not intersect each other, nor any of the subcubes we consider (see Figure~\ref{fig:estimator_setup}). As $A_p$ and $A_q$ are disjoint,  the number of points in in $A_p$ and $A_q$ are independent Poisson random variables.  Further, as they have volume  $\Omega(1/n)$. Hence, $\Prob(|A_p\cap\process_n| >|A_q\cap\process_n|)$  and $\Prob(|A_p\cap\process_n| <|A_q\cap\process_n|)$  are both strictly positive, which implies that any ordering occurs with positive probability.
\begin{figure}
  \centering \includegraphics{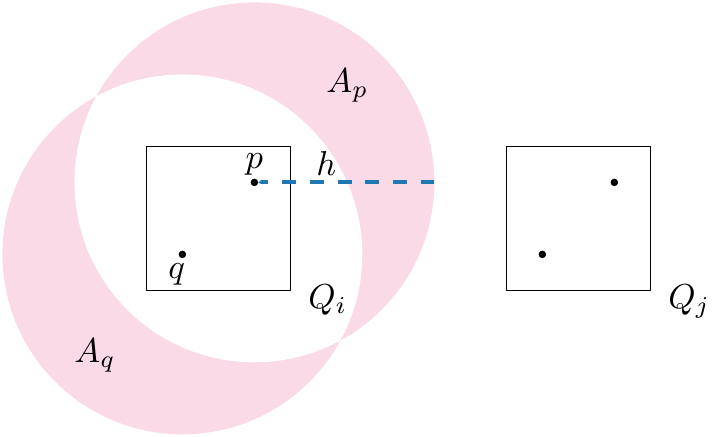}
  \caption{The setup of the proof of~\cref{thm:process_estimation}.}\label{fig:estimator_setup}
\end{figure}
\end{proof}

\begin{remark}
  The argument above requires that $h\rightarrow 0$, as for a
  fixed $h$ the ordering of the points in each $Q_i$ are no longer independent,
  requiring a more delicate analysis. A similar problem occurs for kernels with non-compact support.  %
\end{remark}

We now tackle the case that we construct $\offset_{\filt}(\varphi)$ for
some fixed function $\varphi$, under generic smoothness conditions:

\begin{theorem}\label{thm:process_gamma_fixed}
  Let $\process_n$ be a Poisson point process in $\R^d$ and $\varphi$ a real-valued ($C^2$)Morse function on $\R^d$ with finitely many critical points. The persistence module
$H_k(\offset_{\filt}(\varphi))$ is not interval decomposable
with high probability.
\end{theorem}

The main technical tool is a form of linearization. Essentially, we show that
there exists an isometric transformation (with scaling) that the function will
induce the required ordering on the point set $S$. We first define a technical
condition on the point sets:
\begin{definition}
An ordering on a point set  $S \subset \R^d$  is called $f$-linear if there exists a linear function $f:\R^d \rightarrow \R$ which realizes the ordering on $S$, i.e. $\forall x,y \in S,\; x\preceq y \Leftrightarrow f(x) \leq f(y)$. %
\end{definition}
It is straightforward to check that all of the obstructions to being interval indecomposable we present are $f$-linear. We further make the assumption that $f$ induces a total order on the
point set (which is the case in all our constructions). The main technical
result follows:
\begin{lemma} \label{lem:any_function}
  Given a  $f$-linear ordering on a point set $S$ and a real-valued ($C^2$) Morse function $g$, at any regular point of $g$,  for all sufficiently small scalings there exists an isometry of (scaled) $S$ such that $f$ and $g$ induce the same order on $S$.
\end{lemma}
\begin{proof}
  We assume the values of $f$ at $S$ are distinct. Let $a$ be a regular point of $g$ and index the points in $S$ such that $f(p_i) < f(p_j)$ if $i<j$. Without loss of generality, we can assume $p_0$ is at $a$ which we can take to be the origin and $f(p_0)=0$.

  We require the notion of a \emph{cone ordering}~\cite{{marshall1967order}}. This is a partial ordering on a convex cone $D$ where for $x,y\in D$,  $x\preccurlyeq y$ if and only if  $y-x \in D$.
  Consider the cone induced by the span of the vectors
  $U=\Set{p-q \given p,q \in S, f(p)\geq f(q)}$. By construction, the cone ordering agrees with the order induced by $f$ on $S$. As the values of $f$ are distinct, this cone is acute, i.e.\ the angle between any two points in the cone is less than $\pi$.

  Consider this cone at $p_0$, i.e.\ the origin. If for all $u\in U$,  $\nabla g(p) \cdot u \geq 0$ for all $p\in \mathrm{conv}(S)$, 
  where $\mathrm{conv}(S)$ denotes the convex hull of $S$. Then
  \cite[Corollary 4]{marshall1967order} states that the order induced by $g$ is equivalent to the ordering induced by the cone ordering and hence $f$.

  As $a$ is a regular a point of $g$, $g^{-1}(g(a))$ is a $(d-1)$-dimensional surface with positive reach and hence bounded curvature~\cite{federer1959curvature}. Placing $p_0$ at $a$, rotate $S$ such that $\nabla f(p_0) = \nabla g(a)$. Sufficiently scaling down $S$, we can ensure that the cone spanned by $U$ lies on one side of the tangent plane of $g^{-1}(g(a))$, so $\nabla g(p_0) \cdot u \geq 0$. As the gradient is continuous for $C^2$ functions, for all sufficiently scaled $S$, $\nabla g(p) \cdot u \geq 0$ for all $p\in \mathrm{conv}(S)$, completing the proof.
\end{proof}
An illustration of the construction in the lemma for the $S$ in Figure~\ref{fig:cross_construction} can be seen in Figure~\ref{fig:general_function}. Note that we tilt it slightly to ensure the function values at the points are unique and the red lines indicate an outer bound for the cone $U$.
\begin{figure}
  \centering \includegraphics[width=\textwidth]{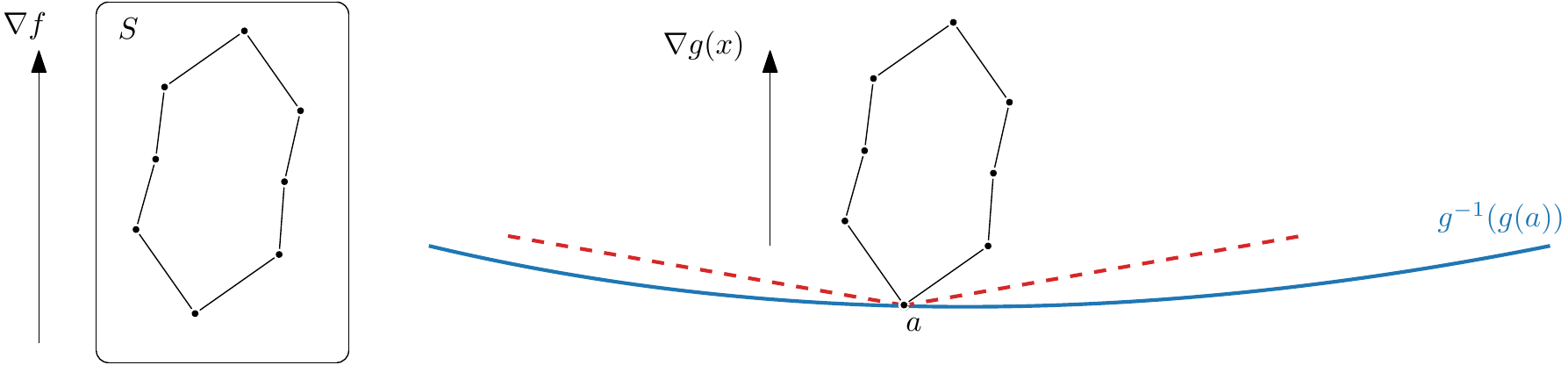}
  \caption{The construction from Lemma \ref{lem:any_function}. We tilt $S$ to ensure the values at the points are unique.  We then place the minimum point at the desired point. If we scale $S$ sufficiently, the cone given by the red lines which contains the set $U$ and will be contained in the superlevel set locally. }\label{fig:general_function}
\end{figure}
We are now ready to finish the proof:
\begin{proof}[Proof of~\cref{thm:process_gamma_fixed}]
Since we assume that $\varphi$ only has a finite number of critical points,  we can place a cube of constant side length such that all points in the cube are regular points of $\varphi$. As it is Morse and $C^2$, this implies that the gradient of $\nabla \varphi(x)$ is non-zero at all points $x$ in the cube. As in Theorem~\ref{thm:poisson}, we can pack $O(n)$ subcubes $Q_i$ of side-length $O(1/n^{1/d})$ within the constant sized cube.
 By Lemma~\ref{lem:any_function}, for each $Q_i$, we can orient $S$ such that the ordering of the vertices induced by $\varphi$ is as in Lemma~\ref{lem:existence_cech_non_interval}. The result follows.
\end{proof}
We note that by~\cref{cor:poisson} this result holds in the case of a non-homogeneous Poisson process, e.g.\ if $\gamma$ is the intensity  of the Poisson process.

\subparagraph{Degree offset bifiltration.} A problem with using density
estimation is that we need to choose a bandwidth parameter. The degree
bifiltrations~\cite{lesnickInteractiveVisualization2D2015} improve on this by
being parameter-free. For a point set $\pointset$ and $r\in\R$, we define the
\deff{degree} of a point $x\in\pointset$ at scale $r$ as the number of points
$y\in\pointset$ with $y\neq x$ such that $\norm{x-y} \leq 2 r$. Intuitively, the degree of a
point estimates its density: higher degree implies higher density. We denote by
$\degree_{r,k}(\pointset)$ the subset of points of $\pointset$ that have at
least degree $k$ at scale $r$. Then, the \deff{degree offset
  bifiltration}
$\degreeoffset_{\filt}(\pointset)\colon\Rplus\times\Rplus^{\op}\to\cTop$ is
given by
$\degreeoffset_{r,k}(\pointset) = \offset_{r}(\degree_{r,k}(\pointset))$.
Similarly for the \deff{degree \v{C}ech bifiltration}
$\degreecech_{\filt}(\pointset)$.

\begin{lemma}\label{lem:existence_degree_cech_non_interval}
  There is a finite point configuration $\pointset$ in $\R^{d}$ of
  $d+3$ points and an
  $\varepsilon > 0$ such that for any
  $\varepsilon$-perturbation $\pointset'$ of $\pointset$,
  $H_{d-1}(\degreeoffset_{\filt}(\pointset'))$ is not interval decomposable.
\end{lemma}
\begin{proof}
  We consider the point set $\pointset$ and the notation of the proof
  of~\cref{lem:existence_cech_non_interval}. Recall that $v$ is the perturbed point of
  $\Sigma_{\delta}$, and $v'$ is that of $\Sigma_{\delta'}^{-}$.

  We construct $S' \coloneqq S\cup\Set{w}$ by placing a new point $w$ close
  to $v'$ in the direction perpendicular to $\sigma$. This distance between $w$ and $v'$ is
  chosen small enough such that, at each scale
  $s\in\Set{r_{\delta'}, r, R_{\delta}}$, the offsets $\offset_{s}(S)$ and
  $\offset_{s}(S')$ are essentially the same: that there is a strong deformation
  retraction from $\offset_{s}(S')$ to $\offset_{s}(S)$. If follows that each
  inclusion $\offset_{s}(S)\hookrightarrow\offset_{s}(S')$ induces an
  isomorphism in $(d-1)$-homology. The same applies to the inclusion
  $\offset_{s}(S\setminus \Set{v})\hookrightarrow\offset_{s}(S'\setminus \Set{v})$.

  Let $P'\subset\Rplus\times\Rplus^{\op}$ be the following subposet, meant to be
  compared to the one in~\cref{fig:regular_simplex_filtration},
  \begin{equation*}
    \begin{tikzcd}[row sep=10pt, column sep=15pt]
      (r_{\delta'}, d) \arrow[r, "<" marking, phantom] & (r, d)\arrow[r, "<" marking, phantom] & (R_{\delta}, d) \\
      & (r, d+1). \arrow[u, "<" marking, phantom] &
    \end{tikzcd}
  \end{equation*}
  Note that at scale $r_{\delta'}$, all points have degree at least $d$: this is clear if
  $d=2$ and, otherwise, $r_{\delta'} > \frac{1}{2}$, so every edge of
  $\Sigma_{\delta}$ and $\Sigma^{-}_{\delta'}$ is present. At scale $r$, $v$ has
  degree $d$ (it is connected only to the $d$ vertices of $\sigma$) and all
  other points have degree at least $d+1$, precisely because we have added $w$.
  By the isomorphisms in $(d-1)$-homology noted above, and noting that
  $\degree_{r_{\delta'}, d}(S') = S'$ and that
  $\degree_{r,d+1}(S') = S'\setminus v$, in parallel to the of
  proof~\cref{lem:existence_cech_non_interval}, we have that the persistence
  module $H_{d-1}(\degreeoffset_{\filt}(S'))$ restricted to $P'$ is isomorphic
  to the left example of~\cref{fig:finite_posets}. The desired statement follows by~\cref{lem:thin-decomp}.

  It is clear that there exists a sufficiently small $\varepsilon > 0$ such that
  the same holds for any $\varepsilon$-perturbation of $S$.
\end{proof}

Using the same argument as before, we obtain:

\begin{theorem}\label{thm:non_interval_degree_cech_random}
  Let $\process_n$ be a Poisson point process in $\R^d$. The persistence module
  $H_k(\degreeoffset_{\filt}(\process_{n}))$ is not interval
  decomposable with high probability.
\end{theorem}

\subsection{Rips bifiltrations}

A variant of the \v{C}ech complex $\cech_{r}(\pointset)$ for a finite point set
$S\subset\R^{d}$ is the \deff{(Vietoris-)Rips complex} $\rips_{r}(\pointset)$,
whose simplexes are those subsets of $\pointset$ of diameter at most $2r$:
\begin{equation}
  \rips_{r}(\pointset) \coloneqq \Set{\sigma\subset\pointset
  \given \diam\sigma \leq 2r},
\end{equation}
where $\diam\sigma$ is the maximum distance between two points in $\sigma$.
The Rips complex assembles into the \deff{Rips filtration} $\rips_{\filt}(\pointset)$
over $\Rplus$, and,
given a function $\gamma\colon S\to\Rplus$, we define the \deff{sublevel Rips
  bifiltration} $\rips_{\filt}(\gamma)\colon\Rplus\to\cSimp$ by
$\rips_{r,s}(\gamma) \coloneqq\rips_{r}(\gamma^{-1}([0, s]))$.

We now comment on how to extend the results of the previous section to the Rips
setting. The construction
in~\cref{lem:existence_cech_non_interval} does not work immediately. Indeed,
in~\cref{fig:regular_simplex_filtration} at $(r, \gamma(v))\in P$, the Rips
complex of such a point set consists of two (filled) triangles: the $1$-skeleton
(the graph given by its vertices and edges)
of $\cech_{r}(\pointset)$ and $\rips_{r}(\pointset)$ coincide, and
$\rips_{r}(\pointset)$ is given by the cliques of this $1$-skeleton.
Focusing on the case of $\R^{2}$, and referring back to the proof
of~\cref{lem:existence_cech_non_interval},~\cref{fig:rips_cross_filtration} proves by
picture the following lemma:
\begin{lemma}\label{lem:existence_rips_non_interval}
  There is a finite point configuration $\pointset$ in $\R^{2}$, a function
  $\gamma\colon \pointset\to\R$ such that $H_{1}(\rips_{\filt}(\gamma))$ is
  not interval decomposable.
  Moreover, the same holds for any $\varepsilon$-perturbation of $S$, for an
  $\varepsilon > 0$ small enough.
\end{lemma}
\begin{figure}
  \centering \includegraphics{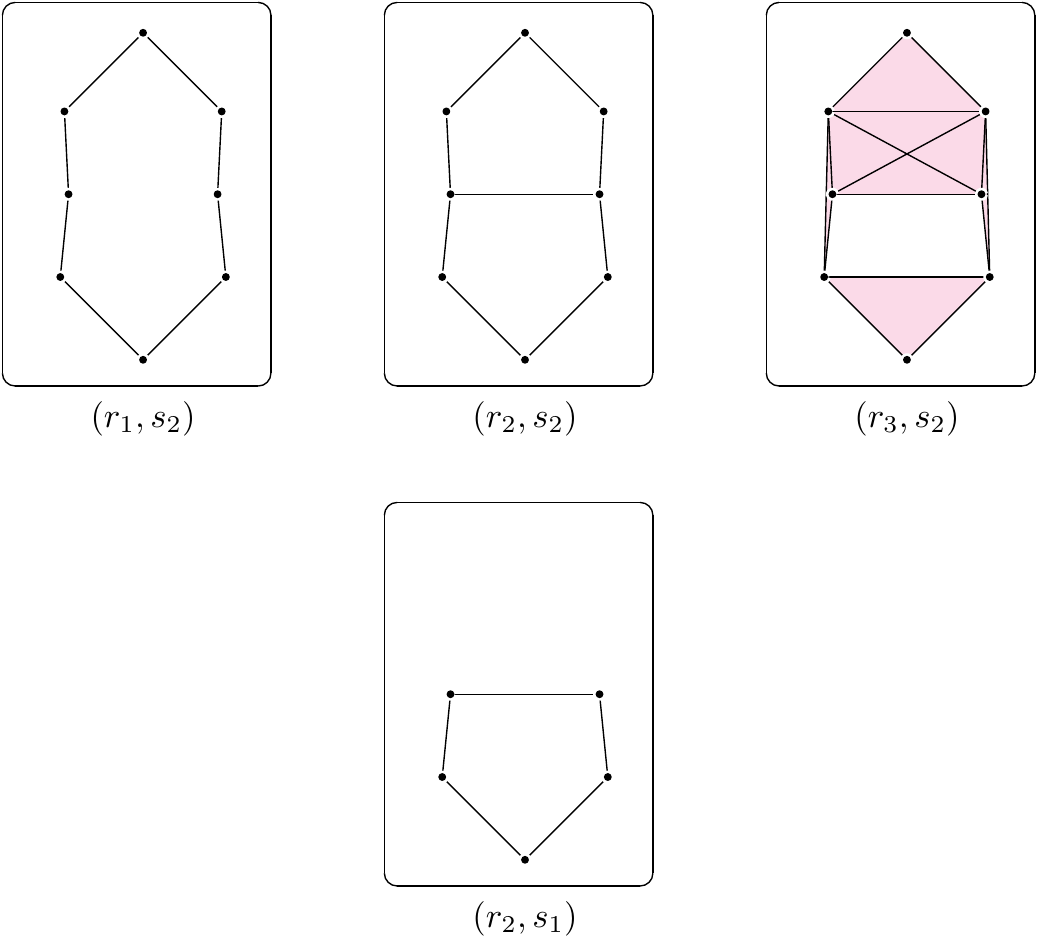}
  \caption{A point set $S$ in the plane and its associated sublevel Rips
    bifiltration restricted to a finite subposet. Shaded triangles are
    the $2$-simplices of the Rips complex.}\label{fig:rips_cross_filtration}
\end{figure}
\begin{figure}
  \centering \includegraphics{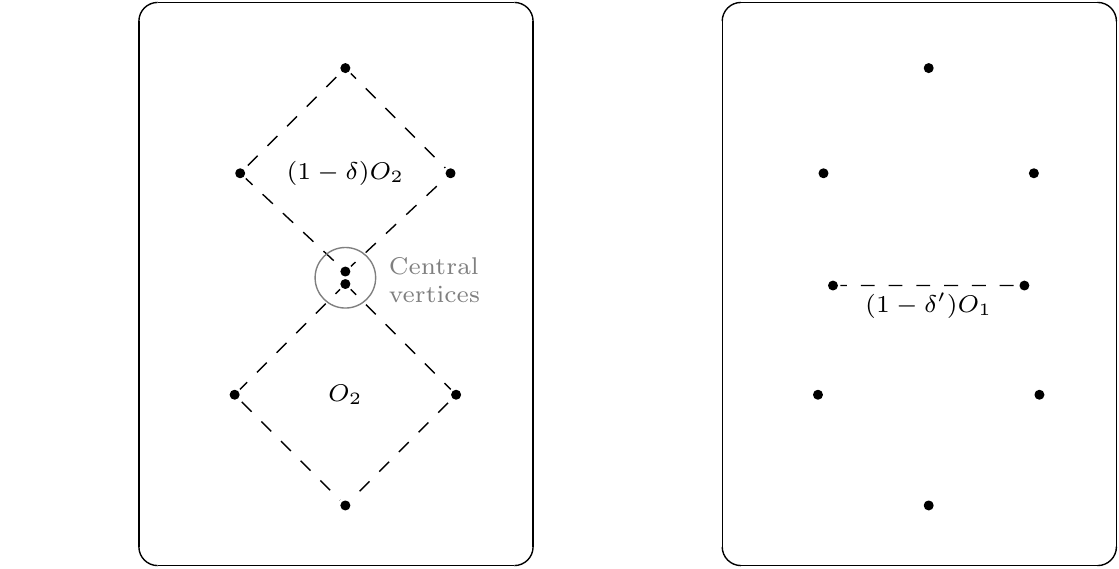}
  \caption{How to construct the point set of~\cref{fig:rips_cross_filtration}.
    On the left, a copy of $(1-\delta)\cross_{2}$ and $\cross_{2}$ lying side by
    side. On the right, the marked central vertices are replaced by
    $(1-\delta')\cross_{1}$.}\label{fig:cross_construction}
\end{figure}

The example $\pointset$ uses as building block the $d$-dimensional
cross-polytope in $\R^{d}$: the convex hull of
the $2d$ points $\Set{\pm e_{1}, \dots, \pm e_{d}}$, where $e_{1}, \dots, e_{d}$ are the
endpoints of the standard basis vectors. We denote the vertex set of the $d$-dimensional
cross-polytope by $\cross_{d}$. Note that, for all $r\in [\sqrt{2}/2, 1)$,
$\rips_{r}(\cross_{d})$ is the boundary of the $d$-cross-polytope, which implies
that
$H_{d-1}(\rips_{r}(\cross_{d}))\isomorphic K$, and, for $r=1$, we have that $\rips_{1}(O_d)$
contains every subset of $O_{d}$.

The point set $\pointset$ consists of a copy of $\cross_{2}$ and a scaled-down
version $(1-\delta) \cross_{2}$, for a small enough $\delta > 0$, lying side by
side, see~\cref{fig:cross_construction}, where we replace the two central
vertices by a copy $(1-\delta')\cross_{1}$, with a sufficiently small
$\delta' > \delta$. An analogous constructions works in higher dimensions
$\R^{d}$, $d > 2$, and $(d-1)$-homology; we omit  the details.

The arguments used to prove the main theorems of~\cref{sec:cech} apply in the
same way to the example of~\cref{lem:existence_rips_non_interval} and thus to
the Rips setting.

\subsection{Zero-dimensional homology and clustering}
\label{sec:zero_dim}

We now handle the case of zero-dimensional homology. Here, the Rips and \v{C}ech
complexes coincide, $H_{0}(\rips_{r}(S)) = H_{0}(\cech_{r}(S))$, since both have
the same $1$-skeleton, and, by the nerve
theorem~\cite{bauerUnifiedViewFunctorial2023a},
$H_{0}(\cech_{r}(S))\isomorphic H_{0}(\offset_{r}(S))$. We use the same strategy as
before:
\begin{lemma}\label{lem:existence_clustering_non_interval}
  There is a finite point configuration $\pointset$ in $\R^{2}$ such
  that $H_{0}(\offset_{\filt}(\gamma))$ is not interval decomposable.
  The same holds for any $\varepsilon$-perturbation of $S$, with
  $\varepsilon > 0$ small enough.
\end{lemma}
\begin{proof}
  Consider the point set $\pointset$ of~\cref{fig:obstruction_uob_dim_0}, and a
  function $\gamma\colon S\to\Rplus$ such that
  $\gamma(A) < \gamma(B) < \gamma(C) < \gamma(D)$. Let $P\subset\Rplus^{2}$ be
  the subposet given by $(0, \gamma(B))$, $(0, \gamma(C))$, $(3.2, \gamma(C))$,
  and $(2.7, \gamma(D))$, which is also described
  in~\cref{fig:obstruction_uob_dim_0}.

  Recalling that $\offset_{r,s}(\gamma)$ is generated by the set of connected
  components, and that, for any two $(r,s)\leq (r',s')$, the map
  $\offset_{(r,s) \to (r',s')}(\pointset)$ is induced by the inclusion of
  connected components, one can see that $H_{0}(\offset_{\filt}(\gamma))$
  restricted to $P$ is isomorphic to and decomposes as:
  \begin{equation*}
   \begin{tikzcd}[ampersand replacement=\&,row sep=25pt, column sep=25pt]
      K^{2} \& \\
      K^{3} \uar{\tinymatrix{1 & 0 & 1 \\ 0 & 1 & 0}}\rar{\tinymatrix{1 & 0 & 0 \\ 0 & 1 & 1}} \& K^{2} \\
      K^{2} \uar{\tinymatrix{1 & 0  \\ 0 & 1 \\ 0 & 0}} \&
    \end{tikzcd}
    \quad
    \isomorphic
    \quad
   \begin{tikzcd}[ampersand replacement=\&,row sep=25pt, column sep=25pt]
      K \& \\
      K \uar{\tinymatrix{1}}\rar{\tinymatrix{1}} \& K \\
      K \uar{\tinymatrix{1}} \&
    \end{tikzcd}
    \quad
    \oplus
    \quad
    \begin{tikzcd}[ampersand replacement=\&,row sep=25pt, column sep=25pt]
      K \& \\
      K^{2} \uar{\tinymatrix{0 & 1}}\rar{\tinymatrix{1 & 0}} \& K \\
      K. \uar{\tinymatrix{1 \\ 1}}\&
    \end{tikzcd}
  \end{equation*}
  \cref{thm:non-interval-crit} yields that
  $H_0(\offset_{\filt}(\pointset))$ is not interval-decomposable, as required.
  Finally, it is clear that the indecomposables do not change by perturbing the
  points by a small enough $\varepsilon$.
\end{proof}
\begin{figure}
\centering
\includegraphics{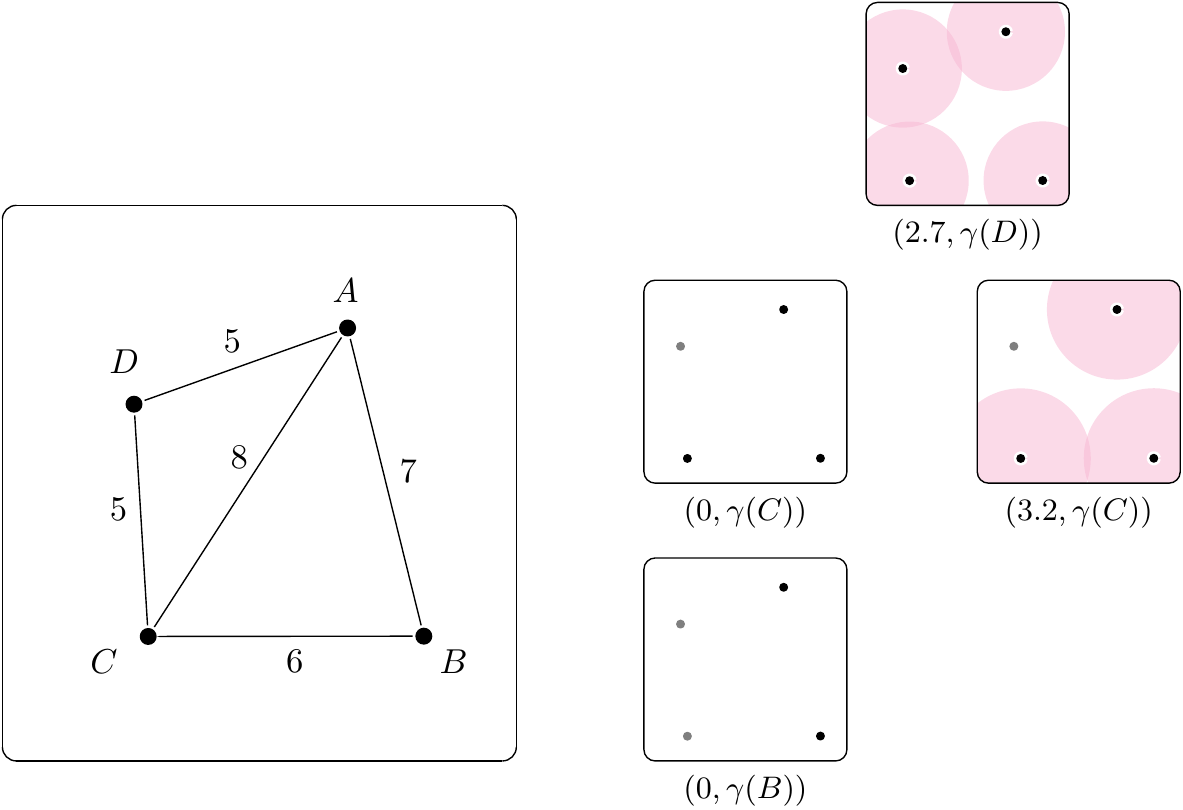}
\caption{The point set $\pointset\subset\R^{2}$ of the proof
  of~\cref{lem:existence_clustering_non_interval}. On the right we describe the
  offset bifiltration restricted to the subposet $P\subset\Rplus^{2}$ of
  the same proof.}\label{fig:obstruction_uob_dim_0}
\end{figure}

The same arguments used in the theorems of~\cref{sec:cech} apply to the
example above.

\subparagraph{Multicover} We now show that multicover bifiltrations
have non-interval-decomposable $0$-homology persistence modules with probability
going to $1$. The \deff{multicover bifiltration}
$\cover_{\filt}(S)\colon \Rplus\times\Rplus^{\op}\to\cTop$ of a finite point set
$S\subset\R^{d}$ is given by
\begin{equation*}
  \cover_{r,k}(S) \coloneqq \Set{y\in\R^{d}\given \norm{y -x} \leq r \text{ for at least $k$ points of $x\in S$}},
\end{equation*}
that is, $\cover_{r,k}$ is the region of $\R^d$ that is covered by at least $k$
of the $r$-balls centered at the points in $S$. Note that
$\cover_{r,1}(S) = \offset_{r}(S)$, and that the multicover bifiltration is sensitive to
density.~\cref{fig:multicover} shows an example for various $r$ and $k$.

\begin{lemma}\label{lem:existence_multicover_non_interval}
  There is a finite point configuration $\pointset$ in $\R^{2}$ such that
  $H_{0}(\cover_{\filt}(S))$ is not interval decomposable.
  The same holds for any $\varepsilon$-perturbation of $S$, with
  $\varepsilon > 0$ small enough.
\end{lemma}
\begin{proof}

  \begin{figure}
    \centering \includegraphics{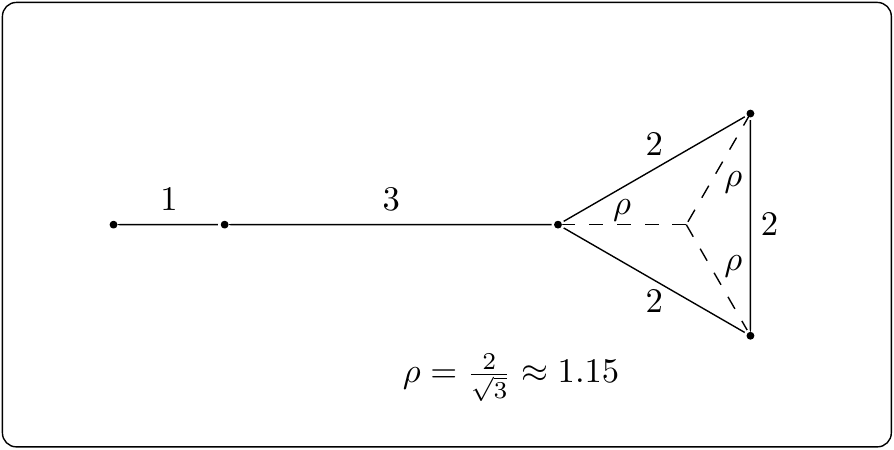}
    \caption{The point set $S\subset\R^{2}$
      of~\cref{lem:existence_multicover_non_interval}. The edges mark the distance
      between the points.}\label{fig:multicover_points}
  \end{figure}

  The point set $S$ is the one described in~\cref{fig:multicover_points}. Let
  $P\subset\Rplus\times\Rplus^{\op}$ be the subposet
  \begin{equation*}
    \begin{tikzcd}[row sep=10pt, column sep=15pt]
      & (r = 2, k = 1) \\
      (r = 0.6, k = 2) \arrow[r, "<" marking, phantom] & (r = 1.2, k = 2) \arrow[u, "<" marking, phantom] \\
      & (r = 1.2, k = 3) \arrow[u, "<" marking, phantom].
    \end{tikzcd}
  \end{equation*}
  We draw $\cover_{p}(S)$ for each $p\in P$ in~\cref{fig:multicover}. Tracking
  how the connected components merge, it can be seen that
  $\restrFilt{H_{0}(\cover_{\filt}(S))}{P}$ is isomorphic to
  \begin{equation*}
    \begin{tikzcd}[ampersand replacement=\&,row sep=15, column sep=15pt]
      \& K \\
      K \rar{\tinymatrix{1 \\ 0}} \& K^{2} \uar[swap]{\tinymatrix{1 & 1}}  \\
      \& K     \uar{\tinymatrix{0 \\ 1}},
    \end{tikzcd}
    \quad\isomorphic\quad
    \begin{tikzcd}[ampersand replacement=\&,row sep=15, column sep=15pt]
      \& K \\
      K \rar{\tinymatrix{1 \\ 1}} \& K^{2} \uar[swap]{\tinymatrix{0 & 1}}  \\
      \& K     \uar{\tinymatrix{0 \\ 1}}.
    \end{tikzcd}
  \end{equation*}
  Comparing this module with the left example
  of~\cref{fig:finite_posets},~\cref{thm:non-interval-crit} yields that
  $H_{0}(\cover_{\filt}(S))$ is not interval decomposable. This holds for any small enough $\eps$-perturbation of the points.
\end{proof}

\begin{figure}
  \centering \includegraphics{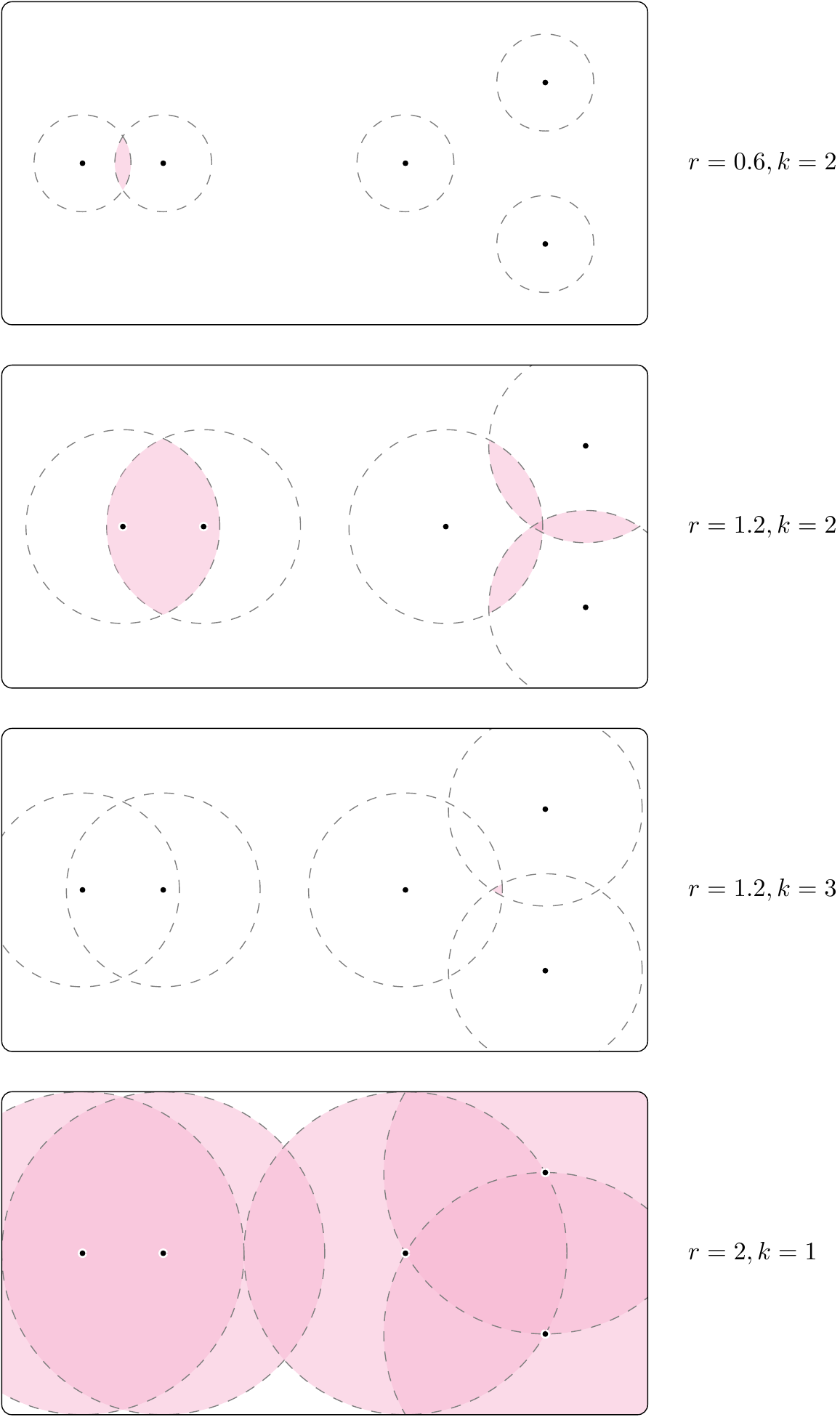}
  \caption{The multicover bifiltration $\cover_{\filt}(S)$ of the point set
    $S\subset\R^{2}$ of~\cref{lem:existence_multicover_non_interval}.}\label{fig:multicover}
\end{figure}

\section{Conclusion}
Building on the empirical observation that in most cases, bifiltrations do not admit a complete decomposition
into intervals, we showed that this occurs with probability going to 1 whenever there is random sampling.
 We focused on the intervals
(or rather thin decomposables), but our technique extends
to prove the presence of arbitrarily complicated indecomposables:
if one can give an $\eps$-stable example for a point set
for which the persistence module of its bifiltration contains that
indecomposable when restricted to a subposet, then a Poisson point process
will contain that indecomposable with probability going to $1$. From these results, there are many interesting directions and
 open questions one can consider.

 First, our examples occur due to noise, and so have short lifetimes. A natural question from the
  perspective of TDA is whether approximations of bifiltrations or persistence modules,
through discretization or the erosion strategy by Bjerkevik~\cite{Bjerkevik},
 would yield a different expected decomposition structure. Furthermore,
 we have only scratched the surface of possible probabilistic questions.
 While we show existence, we would like to understand of the distribution of
 non-interval summands. Within the range of parameters corresponding to ``noise'', what is the most persistent non-interval summand?  Does
  the number of such summands obey a central limit theorem or other universality law, e.g.\ as it is conjectured in the single parameter case for geometric random complexes~\cite{bs-universality}?

  Though we cover several settings, there are some remaining. For example, while
  we covered a simple kernel density estimator, we expect the same results to
  hold for more general kernels, even though the probabilistic elements of the proof
  become much more delicate; we leave this case for future work. Additionally, can
  one show similar results for more general processes, such as binomial or determinantal
  processes where there is dependence between the points, or random functions
  such as Gaussian random fields? In such cases, significantly more advanced
  probabilistic techniques will be needed.

\bibliographystyle{plainurl_fulljournal}

\appendix{}

\section{Proof of~\cref{cor:poisson}}\label{sec:proof_non_homogeneous}

 Denote the cube which satisfies the properties above by $Q\subset \R^d$ and denote its side length by $q$. The point process restricted to $Q$ is also a non-homogenous Poisson process.
Let $f^\uparrow = \max_{x\in Q} f(x)$ and $f^{\downarrow}= \min_{x\in Q} f(x) $. We note that $f^{\downarrow}>0$ by the assumption on $Q$ and $f^{\downarrow}<\infty$ by the assumption on continuity.
We consider the subcube $S$ with side length $(qn)^{-1/d}$.

The probability that the scaled ball $B_i'$ contains
exactly one point is a Poisson variable with rate
$ f^{\downarrow}V(B_i') \leq  \int_{B_i'} f(x)dx \leq f^{\uparrow}V(B_i') $
where $V(B_i')$ denotes the volume of the ball. Then, the probability of exactly
one point occurring in $B_i'$ may be lower bounded by
\[ \Prob(\mu_i \mbox{ points in } B_i') \leq  \frac{\left(f^{\downarrow}V(B_i')\right)^\mu_i e^{-f^{\uparrow}V(B_i') }}{\mu_i!}. \]
By the assumptions on $f$ and the choice of scaling of $B_i'$, the above is lower bounded by a strictly positive constant independent of $n$.

Likewise, the probability of the complement of the scaled points contain no points is a Poisson distributed random variable with rate
$$\int_{T} f(x)dx \leq \int_{Q} f(x)dx \leq f^{\uparrow}V(Q),$$
where $V(Q)$ is the volume of $Q$.
Hence we can lower bound the probability there are no points in the complement by $e^{-f^{\uparrow}V(Q)n}$ which is a strictly positive constant independent of $n$.

The proof follows as in Theorem~\ref{thm:poisson}.

\bibliography{refs}

\begin{thebibliography}{10}

\bibitem{ak-decomposition}
{\'{A}}ngel~Javier Alonso and Michael Kerber.
\newblock Decomposition of zero-dimensional persistence modules via rooted
  subsets.
\newblock In {\em 39th International Symposium on Computational Geometry, SoCG
  2023}, 2023.
\newblock \href {https://doi.org/10.4230/LIPIcs.SoCG.2023.7}
  {\path{doi:10.4230/LIPIcs.SoCG.2023.7}}.

\bibitem{akll-delaunay}
{\'{A}}ngel~Javier Alonso, Michael Kerber, Tung Lam, and Michael Lesnick.
\newblock Delaunay bifiltrations of functions on point clouds.
\newblock In {\em Symposium on Discrete Algorithms SODA 2024}.
\newblock \href {https://doi.org/10.1137/1.9781611977912.173}
  {\path{doi:10.1137/1.9781611977912.173}}.

\bibitem{akp-filtration}
{\'{A}}ngel~Javier Alonso, Michael Kerber, and Siddharth Pritam.
\newblock Filtration-domination in bifiltered graphs.
\newblock In {\em Proceedings of the Symposium on Algorithm Engineering and
  Experiments, {ALENEX} 2023}.
\newblock \href {https://doi.org/10.1137/1.9781611977561.CH3}
  {\path{doi:10.1137/1.9781611977561.CH3}}.

\bibitem{abeny-interval}
Hideto Asashiba, Mickael Buchet, Emerson~G. Escolar, Ken Nakashima, and Michio
  Yoshiwaki.
\newblock On interval decomposability of {2D} persistence modules.
\newblock {\em Computational Geometry}, 105-106:101879, 2022.
\newblock \href {https://doi.org/10.1016/j.comgeo.2022.101879}
  {\path{doi:10.1016/j.comgeo.2022.101879}}.

\bibitem{bauerUnifiedViewFunctorial2023a}
Ulrich Bauer, Michael Kerber, Fabian Roll, and Alexander Rolle.
\newblock A unified view on the functorial nerve theorem and its variations.
\newblock {\em Expositiones Mathematicae}, 41(4):125503, 2023.
\newblock \href {https://doi.org/10.1016/j.exmath.2023.04.005}
  {\path{doi:10.1016/j.exmath.2023.04.005}}.

\bibitem{bauer2023generic}
Ulrich Bauer and Luis Scoccola.
\newblock Generic multi-parameter persistence modules are nearly
  indecomposable, 2023.
\newblock \href {https://arxiv.org/abs/2211.15306} {\path{arXiv:2211.15306}}.

\bibitem{Bjerkevik}
H{\aa}vard~Bakke Bjerkevik.
\newblock Stabilizing decomposition of multiparameter persistence modules,
  2023.
\newblock \href {https://arxiv.org/abs/2305.15550} {\path{arXiv:2305.15550}}.

\bibitem{bl-stability}
Andrew~J. Blumberg and Michael Lesnick.
\newblock Stability of 2-parameter persistent homology.
\newblock {\em Foundations of Computational Mathematics}.
\newblock \href {https://doi.org/10.1007/s10208-022-09576-6}
  {\path{doi:10.1007/s10208-022-09576-6}}.

\bibitem{bobrowski2011distance}
Omer Bobrowski and Robert~J. Adler.
\newblock Distance functions, critical points, and the topology of random Čech
  complexes.
\newblock {\em Homology, Homotopy and Applications}, 16(2):311–344, 2014.
\newblock \href {https://doi.org/10.4310/hha.2014.v16.n2.a18}
  {\path{doi:10.4310/hha.2014.v16.n2.a18}}.

\bibitem{bobrowski2018topology}
Omer Bobrowski and Matthew Kahle.
\newblock Topology of random geometric complexes: a survey.
\newblock {\em Journal of Applied and Computational Topology}, 1(3-4):331--364,
  2018.
\newblock \href {https://doi.org/10.1007/s41468-017-0010-0}
  {\path{doi:10.1007/s41468-017-0010-0}}.

\bibitem{bobrowski2017maximally}
Omer Bobrowski, Matthew Kahle, and Primoz Skraba.
\newblock Maximally persistent cycles in random geometric complexes.
\newblock {\em The Annals of Applied Probability}, pages 2032--2060, 2017.
\newblock \href {https://doi.org/10.1214/16-AAP1232}
  {\path{doi:10.1214/16-AAP1232}}.

\bibitem{bs-universality}
Omer Bobrowski and Primoz Skraba.
\newblock A universal null-distribution for topological data analysis.
\newblock {\em Scientific Reports}, 13(1), July 2023.
\newblock \href {https://doi.org/10.1038/s41598-023-37842-2}
  {\path{doi:10.1038/s41598-023-37842-2}}.

\bibitem{botnan2022consistency}
Magnus~Bakke Botnan and Christian Hirsch.
\newblock On the consistency and asymptotic normality of multiparameter
  persistent {B}etti numbers.
\newblock {\em Journal of Applied and Computational Topology}, December 2022.
\newblock \href {https://doi.org/10.1007/s41468-022-00110-9}
  {\path{doi:10.1007/s41468-022-00110-9}}.

\bibitem{bl-introduction}
Magnus~Bakke Botnan and Michael Lesnick.
\newblock An introduction to multiparameter persistence.
\newblock In {\em 2020 International Conference on Representations of
  Algebras}, page 77–150. EMS Press, November 2023.
\newblock \href {https://doi.org/10.4171/ecr/19/4}
  {\path{doi:10.4171/ecr/19/4}}.

\bibitem{bdk-sparse}
Micka{\"{e}}l Buchet, Bianca~B. Dornelas, and Michael Kerber.
\newblock Sparse higher order {\v{c}}ech filtrations.
\newblock In {\em 39th International Symposium on Computational Geometry, SoCG
  2023, June 12-15, 2023, Dallas, Texas, {USA}}, pages 20:1--20:17, 2023.
\newblock \href {https://doi.org/10.4230/LIPICS.SOCG.2023.20}
  {\path{doi:10.4230/LIPICS.SOCG.2023.20}}.

\bibitem{be-realizations}
Micka{\"{e}}l Buchet and Emerson~G. Escolar.
\newblock Realizations of indecomposable persistence modules of arbitrarily
  large dimensions.
\newblock {\em Journal of Computational Geometry}, 13(1), 2022.
\newblock \href {https://doi.org/10.20382/jocg.v13i1a12}
  {\path{doi:10.20382/jocg.v13i1a12}}.

\bibitem{carlssonTheoryMultidimensionalPersistence2009}
Gunnar Carlsson and Afra Zomorodian.
\newblock The theory of multidimensional persistence.
\newblock {\em Discrete \& Computational Geometry}, 42(1):71--93, 2009.
\newblock \href {https://doi.org/10.1007/s00454-009-9176-0}
  {\path{doi:10.1007/s00454-009-9176-0}}.

\bibitem{cklo-computing}
Ren{\'{e}} Corbet, Michael Kerber, Michael Lesnick, and Georg Osang.
\newblock Computing the multicover bifiltration.
\newblock In {\em 37th International Symposium on Computational Geometry, SoCG
  2021}, pages 27:1--27:17, 2021.
\newblock \href {https://doi.org/10.4230/LIPICS.SOCG.2021.27}
  {\path{doi:10.4230/LIPICS.SOCG.2021.27}}.

\bibitem{dx-computing}
Tamal~K. Dey and Cheng Xin.
\newblock Computing bottleneck distance for {2-D} interval decomposable
  modules.
\newblock In {\em 34th International Symposium on Computational Geometry, SoCG
  2018}, pages 32:1--32:15, 2018.
\newblock \href {https://doi.org/10.4230/LIPIcs.SoCG.2018.32}
  {\path{doi:10.4230/LIPIcs.SoCG.2018.32}}.

\bibitem{dx-generalized}
Tamal~K. Dey and Cheng Xin.
\newblock Generalized persistence algorithm for decomposing multiparameter
  persistence modules.
\newblock {\em Journal of Applied and Computational Topology}, 6:271--322,
  2022.
\newblock \href {https://doi.org/10.1007/s41468-022-00087-5}
  {\path{doi:10.1007/s41468-022-00087-5}}.

\bibitem{federer1959curvature}
Herbert Federer.
\newblock Curvature measures.
\newblock {\em Transactions of the American Mathematical Society},
  93(3):418--491, 1959.
\newblock \href {https://doi.org/10.2307/1993504} {\path{doi:10.2307/1993504}}.

\bibitem{fiedlerMatricesGraphsGeometry}
Miroslav Fiedler.
\newblock {\em Matrices and graphs in geometry}, volume 139 of {\em
  Encyclopedia of Mathematics and its Applications}.
\newblock Cambridge University Press, Cambridge, 2011.
\newblock \href {https://doi.org/10.1017/CBO9780511973611}
  {\path{doi:10.1017/CBO9780511973611}}.

\bibitem{fkr-compression}
Ulderico Fugacci, Michael Kerber, and Alexander Rolle.
\newblock Compression for 2-parameter persistent homology.
\newblock {\em Computational Geometry}, 109:101940, 2023.
\newblock \href {https://doi.org/10.1016/j.comgeo.2022.101940}
  {\path{doi:10.1016/j.comgeo.2022.101940}}.

\bibitem{hnox-refinement}
Yasuaki Hiraoka, Ken Nakashima, Ippei Obayashi, and Chenguang Xu.
\newblock Refinement of interval approximations for fully commutative quivers,
  2023.
\newblock \href {https://arxiv.org/abs/2310.03649} {\path{arXiv:2310.03649}}.

\bibitem{duy2016limit}
Yasuaki Hiraoka, Tomoyuki Shirai, and Khanh~Duy Trinh.
\newblock Limit theorems for persistence diagrams.
\newblock {\em The Annals of Applied Probability}, 28(5):2740--2780, 2018.
\newblock \href {https://doi.org/10.1214/17-AAP1371}
  {\path{doi:10.1214/17-AAP1371}}.

\bibitem{kahle2011random}
Matthew Kahle.
\newblock Random geometric complexes.
\newblock {\em Discrete \& Computational Geometry}, 45(3):553--573, 2011.
\newblock \href {https://doi.org/10.1007/s00454-010-9319-3}
  {\path{doi:10.1007/s00454-010-9319-3}}.

\bibitem{last2017lectures}
G{\"u}nter Last and Mathew Penrose.
\newblock {\em Lectures on the Poisson process}, volume~7.
\newblock Cambridge University Press, 2017.

\bibitem{lesnickInteractiveVisualization2D2015}
Michael Lesnick and Matthew Wright.
\newblock Interactive {{Visualization}} of 2-{{D Persistence Modules}}.
\newblock \href {https://arxiv.org/abs/1512.00180} {\path{arXiv:1512.00180}}.

\bibitem{lw-computing}
Michael Lesnick and Matthew Wright.
\newblock Computing minimal presentations and bigraded betti numbers of
  2-parameter persistent homology.
\newblock {\em SIAM Journal on Applied Algebra and Geometry}, 6(2):267--298,
  2022.
\newblock \href {https://doi.org/10.1137/20M1388425}
  {\path{doi:10.1137/20M1388425}}.

\bibitem{marshall1967order}
A.~W. Marshall, D.~W. Walkup, and R.~J.-B. Wets.
\newblock Order-preserving functions: {A}pplications to majorization and order
  statistics.
\newblock {\em Pacific Journal of Mathematics}, 23:569--584, 1967.

\bibitem{parker-computer}
R.~A. Parker.
\newblock The computer calculation of modular characters (the meat-axe).
\newblock In Michael~D. Atkinson, editor, {\em Computational Group Theory},
  pages 267--274. New York: Academic Press, 1984.

\bibitem{silvermanDensityEstimationStatistics1986}
Bernard~W. Silverman.
\newblock {\em Density {{Estimation}} for {{Statistics}} and {{Data
  Analysis}}}.
\newblock {Springer Netherlands}.
\newblock \href {https://doi.org/10.1007/978-1-4899-3324-9}
  {\path{doi:10.1007/978-1-4899-3324-9}}.

\bibitem{wasserman2004all}
Larry Wasserman.
\newblock {\em All of statistics: a concise course in statistical inference},
  volume~26.
\newblock Springer, 2004.

\bibitem{yogeshwaran2015topology}
D.~Yogeshwaran and Robert~J. Adler.
\newblock On the topology of random complexes built over stationary point
  processes.
\newblock {\em The Annals of Applied Probability}, 25(6):3338--3380, 2015.
\newblock \href {https://doi.org/10.1214/14-AAP1075}
  {\path{doi:10.1214/14-AAP1075}}.

\end{thebibliography}

\end{document}